\documentclass[10pt]{amsart}\oddsidemargin -1mm \evensidemargin -1mm \topmargin -15mm\headheight5mm \headsep 3.5mm \textheight 245mm\textwidth 170mm
\usepackage{amsfonts, amssymb, amsthm, amsmath, euscript, hhline}
\theoremstyle{definition}\newtheorem{theorem}{Theorem}[section]\newtheorem{lemma}[theorem]{Lemma}\newtheorem{remark}[theorem]{Remark}\newtheorem{proposition}[theorem]{Proposition}\newtheorem{corollary}[theorem]{Corollary}\newtheorem{definition}[theorem]{Definition}\newtheorem{Ex}[theorem]{Example}\newtheorem{nota}[theorem]{Notation}
\begin{document}
\title[Primitive ideals of ${\rm U}(\frak{sl}(\infty))$ and the Robinson-Schensted algorithm at infinity]{Primitive ideals of ${\rm U}(\frak{sl}(\infty))$\\ and\\ the Robinson-Schensted algorithm at infinity}
\author{Ivan Penkov, Alexey Petukhov}
\address{Ivan Penkov: Jacobs University Bremen, Campus Ring 1, D-28759, Bremen, Germany}
\email{i.penkov@jacobs-university.de}
\address{Alexey Petukhov: Institute for Information Transmission Problems, Bolshoy Karetniy 19-1, Moscow 127994, Russia}
\email{alex-{}-2@yandex.ru}
\maketitle
\begin{center}
  {\it To Anthony Joseph on the occasion of his 75th birthday}
\end{center}
\begin{abstract} We present an algorithm which computes the annihilator in ${\rm U}(\frak{sl}(\infty))$ of any simple highest weight $\frak{sl}(\infty)$-module $L_\frak b(\lambda)$. 
This algorithm is based on an infinite version of the Robinson-Schensted algorithm.


{\bf AMS Classification 2010:} 17B65, 05E10, 16D60.
\end{abstract}

\section{Background results}\label{Sback}

The description of primitive ideals of the enveloping algebra ${\rm U}(\frak{sl}(n))$ for $n\ge2$ is nowadays a classical result. 
Duflo's Theorem, applied to $\frak{sl}(n)$, claims that, for every fixed Borel subalgebra $\frak b\subset\frak{sl}(n)$, any primitive ideal of ${\rm U}(\frak{sl}(n))$ is the annihilator of a simple $\frak b$-highest weight $\frak{sl}(n)$-module. 
Since (by a well-known generalization of Schur's Lemma) any primitive ideal intersects the centre ${\rm Z}(\frak{sl}(n))$ of ${\rm U}(\frak{sl}(n))$ in a maximal ideal of ${\rm Z}(\frak{sl}(n))$, and since there are only finitely many nonisomorphic simple $\frak b$-highest weight modules with fixed action of ${\rm Z}(\frak{sl}(n))$, Duflo's theorem reduces the problem of classifying primitive ideals to a finite problem. 
Indeed, the Weyl group $S_n$ of $\frak{sl}(n)$ surjects to the set of primitive ideals $I$ with fixed intersection $I\cap{\rm Z}(\frak g)$, and the problem of describing the primitive ideals of ${\rm U}(\frak{sl}(n))$  is equivalent to the problem of describing the fibres of this surjection.

It was Anthony Joseph who solved this latter problem by reducing it to the Robinson-Schensted algorithm.

The purpose of our current paper is to establish a combinatorial counterpart of Joseph's result for the infinite-dimensional Lie algebra $\frak{sl}(\infty)$. 
More precisely, we provide an algorithm for computing the primitive ideal of any simple highest weight $\frak{sl}(\infty)$-module. 
This algorithm is our proposed ``Robinson-Schensted algorithm at infinity''.

We start with a brief survey of previous results on the primitive ideals of ${\rm U}(\frak g_\infty)$ for direct limit Lie algebras $\frak g_\infty$, putting in this way the current paper into context. 

The Lie algebra $\frak{sl}(\infty)$ is defined as the direct limit of an arbitrary chain of embeddings
$$\frak{sl}(2)\hookrightarrow\frak{sl}(3)\hookrightarrow\frak{sl}(4)\hookrightarrow...$$
More generally, one may consider an arbitrary chain of embeddings of simple Lie algebras
\begin{equation}\frak g_1\hookrightarrow\frak g_2\hookrightarrow...\hookrightarrow\frak g_n\hookrightarrow\frak g_{n+1}\hookrightarrow...\label{Ech}\end{equation}
and its direct limit $\frak g_\infty=\varinjlim \frak g_n$. 

An embedding $\frak g_i\hookrightarrow\frak g_{i+1}$ as in~(\ref{Ech}) is {\it diagonal} if the branching rule for the natural $\frak g_{i+1}$-modules (the nontrivial simple $\frak g_{i+1}$-modules of minimal dimension) involves only natural and trivial modules over $\frak g_i$. The direct limits of chains of diagonal embeddings are known as {\it diagonal Lie algebras} and are classified by A. Baranov and A. Zhilinskii~\cite{BZh}. 
Furthermore, diagonal Lie algebras can be split into {\it nonfinitary diagonal} Lie algebras and {\it finitary} Lie algebras, the latter being (up to isomorphism) just three Lie algebras: $\frak{sl}(\infty), \frak o(\infty), \frak{sp}(\infty)$. 
The finitary  Lie algebras $\frak g_\infty$ are defined as the direct limits of chains (\ref{Ech}) where $\frak g_n=\frak{sl}(n+1)$, $\frak o(n),$ $\frak{sp}(2n)$, respectively. 

The classification problem for nondiagonal Lie algebras $\frak g_\infty$ appears to be wild. Nevertheless, one can make the following strong statement about primitive ideals in ${\rm U}(\frak g_\infty)$:

{\it If $\frak g_\infty$ is nondiagonal, i.e., there are infinitely many nondiagonal embeddings in the chain~(\ref{Ech}), the only proper two-sided ideals in ${\rm U}(\frak g_\infty)$ are the augmentation ideal and the zero ideal. }

This statement is known as Baranov's conjecture and is proved in~\cite{PP1}. 

For nonfinitary diagonal Lie algebras $\frak g_\infty$, a classification of two-sided ideals is obtained by A. Zhilinskii~\cite{Zh3}. 
Here there are two-sided ideals $I$ different from the augmentation ideal, however a characteristic feature of this case is that all quotients ${\rm U}(\frak g_\infty)/I$ are locally finite dimensional. 
By definition, this means that the quotients ${\rm U}(\frak g_n)/({\rm U}(\frak g_n)\cap I)$ are finite dimensional. 
A similar result has been established in the recent paper~\cite{PeS} also for the Witt Lie algebra (which is not a direct limit of finite-dimensional Lie algebras), and this leads us to the thought that the above results might extend to a larger class  of infinite-dimensional Lie algebras. 
That could be a subject of future research.

None of the above results apply to the three finitary Lie algebras $\frak{sl}(\infty), \frak{o}(\infty), \frak{sp}(\infty)$.
The problem of classifying primitive ideals in the enveloping algebras ${\rm U}(\frak{sl}(\infty))$, ${\rm U}(\frak{o}(\infty))$ and ${\rm U}(\frak{sp}(\infty))$ has been open for some time, and was recently solved in~\cite{PP5} for ${\rm U}(\frak{sl}(\infty))$. 
Here is a brief history of the problem. 
It was posed by A. Zalesskii, who saw it as a problem analogous to classifying primitive (and two-sided) ideals in the group algebra of $S_\infty$. 
Indeed, the latter problem admits a relatively straightforward combinatorial solution, and suggests a method for constructing primitive ideals of ${\rm U}(\frak{sl}(\infty))$. 
One considers {\it coherent local systems of simple $\frak{sl}(n)$-modules}: such coherent local systems, {\it c.l.s.} for short, consist of nonempty sets of isomorphism classes $[L_n^\alpha]$ of simple finite-dimensional $\frak{sl}(n)$-modules $L_n^\alpha$ for each $n\ge2$, such that each $\frak{sl}(n)$-module $L_n^{\alpha_0}$ branches over $\frak{sl}(n-1)$ as a sum of $\frak{sl}(n-1)$-modules among $L_{n-1}^{\alpha}$, and every $L_{n-1}^\alpha$ arises from a suitable $L_n^{\alpha_0}$. 
The joint annihilator in ${\rm U}(\frak{sl}(\infty))$ of such a c.l.s. (i.e. the union over $n$ of all joint annihilators $\cap_\alpha{\rm Ann}_{{\rm U}(\frak{sl}(n))}L_n^\alpha$ is a two-sided ideal of ${\rm U}(\frak{sl}(\infty))$. 
Furthermore, one can prove that if a c.l.s. is {\it irreducible}, i.e., is not a proper union of two sub-c.l.s., then its annihilator is a primitive ideal. 
As an important step in Zalesskii's program, A. Zhilinskii classified all c.l.s. (and, in particular, irreducible c.l.s.). 
Unfortunately, Zhilinskii's work is not widely available as his main paper~\cite{Zh3} is a preprint in Russian. 
We have given a summary of Zhilinskii's classification of c.l.s. in our survey paper~\cite{PP3}, see also~\cite{PP1}.

As a next step, we determined in~\cite{PP1} which simple c.l.s. have the same annihilator, and completed in this way the classification of primitive ideals of ${\rm U}(\frak{sl}(\infty))$ arising from c.l.s.. 
We call these ideals {\it integrable} primitive ideals (an equivalent definition is given in~\cite{PP1},~\cite{PP3}).

The next step was made in our work~\cite{PP5} where we proved  that any primitive ideal of ${\rm U}(\frak{sl}(\infty))$ is integrable, providing finally a classification of primitive ideals of ${\rm U}(\frak{sl}(\infty))$. 
The proof is based on three pillars:
our study of associated pro-varieties of primitive ideals in~\cite{PP1}, 
Joseph's original classification of primitive ideals in ${\rm U}(\frak{sl}(n))$,
and certain new combinatorial facts relating ``precoherent local systems'' of representations of $\frak{sl}(n)$ for $n\ge1$ to coherent local systems introduced above. 
These latter facts use heavily the Gelfand-Tsetlin branching rule.

The final result is as follows:

{\it Primitive ideals of ${\rm U}(\frak{sl}(\infty))$ are naturally parametrized by quadruples
$$(r, g, X, Y)$$

where $r, g$ are nonnegative integers and $X, Y$ are Young diagrams.}

The integer $r$ is the {\it rank} and represents the associated pro-variety of a primitive ideal, see~\cite{PP1}. 
The integer $g$ is the {\it Grassmann number}. 
We call it so as it arises naturally from direct limits of exterior powers of defining $\frak{sl}(n)$-modules, i.e., of direct limits of fundamental $\frak{sl}(n)$-modules. 
More precisely, a {\it semiinfinite fudamental $\frak{sl}(\infty)$-module} is a direct limit of fundamental $\frak{sl}(n)$-modules whose degrees and codegrees both tend to infinity (there are uncountably many nonisomorphic semiinfinite modules), see~\cite{GP}. 
The annihilators of all semiinfinite fundamental $\frak{sl}(\infty)$-modules coincide, and the corresponding ideal is labeled by $(0, 1, \emptyset, \emptyset)$. 

Finally, the Young diagrams ${X}, {Y}$ also arise in a staighforward manner: the primitive ideal with coordinates $(0, 0, {X}, {Y})$ is the annihilator of the simple tensor module $V_{{X}, {Y}}$; 
this module is defined as the socle  of the tensor product $S_{{X}}(V)\otimes S_{{Y}}(V_*)$ where $V$ and $V_*$ are the two defining representations of $\frak{sl}(\infty)$ (finitary column vectors and finitary row vectors) and $S_Z(\cdot)$ is the Schur functor associated to a Young diagram $Z$, see~\cite{DPSn}.

An essential difference with the case of $\frak{sl}(n)$ is that the annihilator in ${\rm U}(\frak{sl}(\infty))$ of most simple $\frak{sl}(\infty)$-modules is equal to zero. 
Therefore one can think of simple $\frak{sl}(\infty)$-modules with nonzero annihilators as {\it small}. 
Examples of small simple modules are the above mentioned modules $V_{{X}, {Y}}$, semiinfinite fundamental representations, and also direct limits of growing symmetric powers of defining representations of $\frak{sl}(n)$ for $n\to\infty$. 
A small simple $\frak{sl}(\infty)$-module does not need to be integrable, i.e., does not need to be a direct limit of finite-dimensional $\frak{sl}(n)$-modules for $n\to\infty$. 
For instance, in~\cite{GP} it is shown that any simple weight $\frak{sl}(\infty)$-module with bounded weight multiplicities is small. 
However, our classification of primitive ideals implies that the annihilator of any small simple $\frak{sl}(\infty)$-module is also the annihilator of a, possibly nonisomorphic, simple integrable $\frak{sl}(\infty)$-module. 
This is a  truly infinite-dimensional effect.

\section{Our goal in the present paper}

We are now ready to explain the purpose of the paper. 
Despite the fact that primitive ideals of ${\rm U}(\frak{sl}(\infty))$ are classified, the existing literature does not explain how to compute the annihilator of an arbitrary simple highest weight module $L_\frak b(\lambda)$, i.e.,  how to find the quadruple $(r, g, {X}, {Y})$ corresponding to the ideal ${\rm Ann}_{{\rm U}(\frak{sl}(\infty))}L_\frak b(\lambda)$, for a given splitting Borel subalgebra $\frak b\subset\frak{sl}(\infty)$ and a character $\lambda$ of $\frak b$. 
Solving this problem is our aim in the present work. 
In the case of $\frak{sl}(n)$, the analogous problem is solved by applying the Robinson-Schensted algorithm to the weight $\lambda+\rho$, and below we present the corresponding ``infinite version'' of this algorithm. 

In the work~\cite{PP2} we have established an important preliminary result: we have found a necessary and sufficient condition on the pair $(\frak b, \lambda)$ for the annihilator ${\rm Ann}_{{\rm U}(\frak{sl}(\infty))}L_\frak b(\lambda)$ to be nonzero. 
Recall that a splitting Borel subalgebra containing a fixed splitting Cartan subalgebra (for instance, the diagonal matrices in $\frak{sl}(\infty)$) is given by an arbitrary total order $\prec$ on a countable set, see~\cite{DP}. 
We denote this set by $\Theta$. 
Theorem 3.1 in~\cite{PP2} asserts that ${\rm Ann}_{{\rm U}(\frak{sl}(\infty))}L_\frak b(\lambda)\ne0$ if and only if $\Theta$ can be split as a finite disjoint union
$$\Theta=\Theta_1\sqcup...\sqcup \Theta_k$$
such that $i\prec j$ for any pair $i\in \Theta_s, j\in \Theta_t$ with $s<t$, and the restriction of $\lambda$ to $\Theta_s$ is a  constant $\lambda(s)$ for any $s<k$, satisfying $\lambda(s)-\lambda(t)\in\mathbb Z$ if both $\Theta_s, \Theta_t$ are infinite. 

The above makes it clear that in order to compute ${\rm Ann}_{{\rm U}(\frak{sl}(\infty))}L_\frak b(\lambda)$ we need to provide an algorithm which transforms a given pair $(\frak b, \lambda)$, where $\frak b$ is a splitting Borel subalgebra and $\lambda$ is a weight such that ${\rm Ann}_{{\rm U}(\frak{sl}(\infty))} L_\frak b(\lambda)\ne 0,$ to the quadruple corresponding to the primitive ideal $(r, g, {X}, {Y})$ of ${\rm Ann}_{{\rm U}(\frak{sl}(\infty))}L_\frak b(\lambda)$. 
This is precisely what we do: we construct a version of the Robinson-Schensted algorithm which performs the above task.

{\bf Acknowledgements.} 
We thank Professors Maria Gorelik and Anna Melnikov for inviting us to participate and present the results of this paper at the sister conferences in Israel celebrating Anthony~Joseph's 75th birthday. 
We are grateful to Professor Anthony Joseph for his kind attention to our work. 
Both authors have been supported  in part by DFG grant PE 980/6-1, and the second author has been supported in part by~RFBR grant 16-01-00818.

\section{Preliminaries}
\subsection{Robinson-Schensted algorithm and $\frak{sl}(n)$}\label{Sprsl}
\subsubsection{Notation}
We fix an algebraically closed field $\mathbb F$ of characteristic 0. 
If $V$ is a vector space over $\mathbb F$, we set $V^*={\rm Hom}_{\mathbb F}(V, \mathbb F)$. 
All ideals in associative $\mathbb F$-algebras are assumed to be two-sided. 
We use the notions of Young diagrams and partitions as synonyms; when writing a Young diagram as a partition $(p_1\ge p_2\ge...\ge p_n>0)$, the integers $p_i$ are the row lengths of the diagram. 

We identify $\frak{sl}(n)$ with the set of traceless $n\times n$-matrices. 
The elementary matrices $$e_{i, j}\hspace{10pt}{\rm~for~}1\le i\ne j\le n,\hspace{10pt}e_{i, i}-e_{i+1, i+1}\hspace{10pt}{\rm~for~}1\le i\le n-1$$ 
form a basis of $\frak{sl}(n)$. 
We fix the Cartan subalgebra $\frak h_n$ of diagonal matrices and the Borel subalgebra $\frak b_n$ of upper triangular matrices. 
To any linear function $\lambda\in\frak h_n^*$ we attach the linear map
$$\lambda': \frak b_n\to\mathbb F,\hspace{10pt}e_{ij}\mapsto 0{\rm~for~}i\ne j,{\rm~and~}\lambda'|_{\frak h_n}=\lambda.$$
We denote by $\mathbb F_\lambda$ the one-dimensional $\frak b_n$-module defined by $\lambda'$. 
Set
$$M(\lambda):={\rm U}(\frak{sl}(n))\otimes_{{\rm U}(\frak b_n)}\mathbb F_\lambda.$$
Let $L(\lambda)$ be the unique simple quotient of $M(\lambda)$, and
$$I(\lambda):={\rm Ann} L(\lambda).$$

We identify the vector space $\mathbb F^n$ with the space of functions
$$f: \{1,..., n\}\to \mathbb F.$$
For any function $f\in\mathbb F^n$ there exists a unique $\lambda_f\in\frak h_n^*$ such that
$$\lambda_f(e_{ii}-e_{jj})=f(i)-f(j).$$
Therefore to any function $f\in\mathbb F^n$ we can attach the primitive ideal
$$I(f):=I(\lambda_f)\subset{\rm U}(\frak{sl}(n)).$$
The Weyl group $W_n$ of the pair $(\frak{sl}(n), \frak h_n)$ is the symmetric group $S_n$, and its action on $\frak h_n^*$ is induced by its action on $\mathbb F^n$ via permutations. 
The shifted action of $S_n$ on $\mathbb F^n$, denoted by $\sigma\cdot f$, is defined as $$\sigma\cdot f:=\sigma(f+\rho_n)-\rho_n$$ where $\rho_n:=(-1, -2,..., -n)$.  
\subsubsection{Joseph's description of primitive ideals} Let ${\rm Prim}{\rm U}(\frak g)$ be the set of primitive ideals of ${\rm U}(\frak g)$. 
Duflo's Theorem implies that the map
$$\psi:\mathbb F^n\to{\rm Prim} {\rm U}(\frak g), \hspace{10pt}f\mapsto I(f)$$
is surjective. 
A description of ${\rm Prim}{\rm U}(\frak g)$, based on the description of the fibres of $\psi$, is due to Joseph~\cite{Jo}, see also~\cite{BJ, BV1, Jo, Jo2}.

As a first step of this description, one attaches to $f\in\mathbb F^n$ a subgroup $W_n(f)\subset W_n$ called {\it the integral Weyl subgroup of $f$}. 
The subgroup $W_n(f)$ is a parabolic subgroup of $S_n$, and therefore is a product of permutation groups.  
As a second step, one defines an element $w(f)\in W(\lambda)$. 
In the regular case, this element $w(f)$ produces $f$ from its dominant representative. 
For the singular case we refer the reader to~\cite{BJ}. 
The third step consists of applying the Robinson-Schensted algorithm to each factor of the element $w(f)$ with respect to the decomposition of $W_n(f)$ as a direct product of symmetric groups. 
For each factor of $w(f)$ this algorithm produces a pair of semistandard  Young tableaux called {\it recording} tableau and {\it insertion} tableau. 

The original result of Joseph~\cite{Jo1} claims that $\psi(f_1)=\psi(f_2)$ if and only if

1) $f_1$ and $f_2$ define the same character of ${\rm Z}(\frak{sl}(n))$, i.e., there exists $k\in\mathbb F$ and a permutation $\sigma\in S_n$ such that $\sigma\cdot f_1=f_2+k$,

2) the recording tableau of each factor of $w(f_1)\in W_n(f_1)$ coincides with the recording tableau of the corresponding factor of $w(f_2+k)$ under $\sigma$.

For the purpose of considering the limit $n\to\infty$, it is convenient to restate Joseph's result in terms of $f$ only without referring to $w(f)$. 
We do this in Theorem~\ref{Tjo} below. 
\subsubsection{Admissible interchanges}\label{SSkm} For $a, b\in\mathbb F$ we write $a>_{\mathbb Z}b$ whenever $a-b\in\mathbb Z_{>0}$. 
The notations $a<_\mathbb Zb, a\ge_{\mathbb Z}b, a\le_\mathbb Zb$ have similar meaning.

Let $f_1, f_2\in\mathbb F^n$. 
We say that $f_1$ and $f_2$ are {\it connected by the $i$th admissible interchange} if
$$f_1(j)=f_2(j), j\ne i, i+1,\hspace{10pt} f_1(i)=f_2(i+1), f_1(i+1)=f_2(i), 1\le i\le n-1,$$
and one of the following conditions is satisfied:

1) $f_1(i+1)-f_1(i)\not\in\mathbb Z,$

2) $i\le n-2$ and $f_1(i+1)>_{\mathbb Z} f_1(i+2)\ge_{\mathbb Z} f_1(i)$,

2$'$) $i\le n-2$ and $f_1(i)>_{\mathbb Z} f_1(i+2)\ge_{\mathbb Z} f_1(i+1)$,

3) $i\ge 2$ and $f_1(i+1)\ge_{\mathbb Z} f_1(i-1)>_{\mathbb Z} f_1(i)$,

3$'$) $i\ge 2$ and $f_1(i)\ge_{\mathbb Z} f_1(i-1)>_{\mathbb Z} f_1(i+1)$.\\
It can be easily checked that $f_1$ is connected with $f_2$ by the $i$th admissible interchange if and only if $f_2$ is connected with $f_1$ by the $i$th admissible interchange. 
These admissible interchanges are known in the context of the Robinson-Schensted algorithm, see Theorem~\ref{Tjo} below. 

We say that $f_1$ and $f_2$ are connected by the {\it shifted $i$th admissible interchange} if the sequences $f_1+\rho_n{\rm~and~}f_2+\rho_n$ are connected by the $i$th admissible interchange. 
\subsubsection{A version of Robinson-Schensted algorithm for finite sequences}\label{SSrskf}
The Robinson-Schensted algoithm is a classical object of 20th century mathematics and has different versions.
As a reference for ``the standard algorithm'' we use~\cite{Knu}. 
This algorithm works with a finite sequence of nonrepeating integers,  however we note that one can apply the standard algorithm to any nonrepeating finite sequence of elements of a totally ordered set $(\mathcal S, \prec)$.  
The output of this procedure consists of a Young tableau filled by elements of $\mathcal S$ (recording tableau) and a Young tableau of the same shape filled by positive integers (insertion tableau). 
The recording tableau is standard with respect to $\prec$ and the insertion tableau is standard with respect to $<$. 

If a sequence consists of elements of several distinct totally ordered sets $(\mathcal S_i, \prec_i)$, we can split the sequence into subsequences of elements of $\mathcal S_i$ (one for each set) and apply the algorithm separately to such sequences. 
The output consists of a collection of pairs of tableaux -- one pair per set $\mathcal S_i$. 

In our case, the totally ordered sets $(\mathcal S_i, \prec)$ will be of the form $(a+\mathbb Z)\times\mathbb Z$, $a\in\mathbb F$, with the order 
\begin{equation}(a, i)\prec (b, j)\iff[(a>_\mathbb Zb){\rm~or~}(a=b, i>j)].\label{Eord}\end{equation}

Let $f_1,..., f_n\in \mathbb F$ be a finite sequence. We attach to $f_1,..., f_n$ the sequence
\begin{equation}\label{Erho}(f_1, 1),..., (f_n, n),\end{equation}
and split~(\ref{Erho}) into totally ordered subsets of $(a+\mathbb Z)\times\mathbb Z$ as above. 
We then apply the standard Robinson-Schensted algorithm to (\ref{Erho}). 
The output consists of a collection of pairs of tableaux. 
The recording tableau in a pair is filled by $\{(f_i, i)\}_{1\le i\le n}$ and the insertion tableau is filled by $1,..., n$. 
As a last step we replace the pairs $(f_i, i)$ in all recording tableaux by $f_i$ and discard all insertion tableaux. 
The resulting tableaux have strictly decreasing rows and nonincreasing columns (the corner of a tableau being in the upper-left position). 
This is a consequence of the inequality inversion in the left and right-hand sides of formula~(\ref{Eord}).

In what follows, by {\it RS-algorithm}, we mean the above procedure. 
We denote by $RS(f_1,..., f_n)$ its output. 
We set also $$J(f_1,..., f_n):=RS(f_1-1,..., f_n-n);$$
$J(f_1,..., f_n)$ reflects the shift of $f_1,..., f_n$ by ``$\rho$''. 
\begin{Ex}Consider the sequence $3, 4, 4, \alpha,$ where $\alpha\notin\mathbb Z$. We have
$$J(3, 4, \alpha, 5)=RS(2, 2, \alpha-3, 1).$$
Next, we attach to the sequence $2, 2, \alpha-3, 1$ the sequence
\begin{equation}(2, 1), (2, 2), (\alpha-3, 3), (1, 4)\label{Eprers}\end{equation}
of elements of $\mathbb F\times\mathbb Z$. We have
\begin{equation}(2, 2)\prec (2, 1)\prec (1, 4),\label{E222114}\end{equation}
and the element $(\alpha-3, 3)$ is incomparable with the elements of~(\ref{E222114}). 
We apply the RS-algorithm to the sequence~(\ref{Eprers}) step-by-step from left to right:
$$(2, 1)\mapsto\{(\begin{array}{c}
\hline\multicolumn{1}{|c|}{(2, 1)}\\
\hhline{-}
\end{array}, \begin{array}{c}
\hline\multicolumn{1}{|c|}{1}\\
\hhline{-}
\end{array})\},\hspace{10pt}
((2, 1), (2, 2))\mapsto
\{(\begin{array}{c}
\hline\multicolumn{1}{|c|}{(2, 2)}\\
\hline\multicolumn{1}{|c|}{(2, 1)}\\
\hhline{-}
\end{array},
\begin{array}{c}
\hline\multicolumn{1}{|c|}{2}\\
\hline\multicolumn{1}{|c|}{1}\\
\hhline{-}
\end{array})
\},$$
$$((2, 1), (2, 2), (\alpha-3, 3))\mapsto
\{(\begin{array}{cc}
\hline\multicolumn{1}{|c|}{(2, 2)}\\
\hline\multicolumn{1}{|c|}{(2, 1)}\\
\hhline{-}
\end{array},
\begin{array}{cc}
\hline\multicolumn{1}{|c|}{2}\\
\hline\multicolumn{1}{|c|}{1}\\
\hhline{-}
\end{array}),  
(\begin{array}{c}
\hline\multicolumn{1}{|c|}{(\alpha-3, 3)}\\
\hhline{-}
\end{array},
\begin{array}{c}
\hline\multicolumn{1}{|c|}{3}\\
\hhline{-}
\end{array})
\},
$$
$$((2, 1), (2, 2), (\alpha-3, 3), (1, 4))\mapsto\{(\begin{array}{cc}
\hline\multicolumn{1}{|c|}{(2, 2)}&\multicolumn{1}{|c|}{(1, 4)}\\
\hline\multicolumn{1}{|c|}{(2, 1)}&\\
\hhline{-~}
\end{array}, 
\begin{array}{cc}
\hline\multicolumn{1}{|c|}{2}&\multicolumn{1}{|c|}{4}\\
\hline\multicolumn{1}{|c|}{1}&\\
\hhline{-~}
\end{array}), 
(\begin{array}{c}
\hline\multicolumn{1}{|c|}{(\alpha-3, 3)}\\
\hhline{-}
\end{array}, 
\begin{array}{c}
\hline\multicolumn{1}{|c|}{3}\\
\hhline{-}
\end{array})\}.
$$
The result is
$$J(3, 4, \alpha, 5)=\{\begin{array}{cc}
\hline\multicolumn{1}{|c|}{2}&\multicolumn{1}{|c|}{1}\\
\hline\multicolumn{1}{|c|}{2}&\\
\hhline{-~}
\end{array}, \begin{array}{c}
\hline\multicolumn{1}{|c|}{\alpha-3}\\
\hhline{-}
\end{array}\}.
$$
\end{Ex}
\begin{theorem}[{An~equivalent~form~of~Joseph's~theorem}]\label{Tjo}The following conditions are equivalent for sequences $f_1,..., f_n\in\mathbb F$, $f_1',..., f_n'\in\mathbb F$:

1) $I(f_1,..., f_n)=I(f_1',..., f_n')$,

2) $\exists k\in\mathbb F: J(f_1,..., f_n)=J(f_1'+k,..., f_n'+k)$,

3) there exists $k\in\mathbb F$ so that the sequences $$f_1,..., f_n{\rm~and~}f_1'+k,..., f_n'+k$$ are connected by a series of shifted admissible interchanges.\end{theorem}
\begin{proof}This is implied by the results of~\cite{Jo1} and~\cite[Exercise~4~on~page~65]{Knu}.\end{proof}
In what follows, it will be convenient to encode Young tableaux via sequences. 
\begin{nota}To a Young tableau $T$ with $n$ boxes filled by elements of $a+\mathbb Z$ we attach the sequence ${\rm seq}(T)\in\mathbb F^n$ which consists of the rows of $T$ ordered in the inverse lexicographical order (shorter rows come first; among rows of equal length, rows with smaller first element come first).\end{nota}
It is straightforward to check that $$RS({\rm seq}(T))=T.$$
This implies that $T$ can be encoded by ${\rm seq}(T)$. If $T_1,..., T_s$ is a sequence of tableaux we set ${\rm seq}(T_1,..., T_s)$ to be the concatenation of the sequences ${\rm seq}(T_1),...,{\rm seq}(T_s)$.
\begin{Ex}If $T=
\begin{array}{ccc}
\hline\multicolumn{1}{|c|}{4+a}&\multicolumn{1}{|c|}{2+a}&\multicolumn{1}{|c|}{1+a}\\
\hline \multicolumn{1}{|c|}{4+a}&\multicolumn{1}{|c|}{1+a}&\\
\hhline{--~}\multicolumn{1}{|c|}{4+a}&\multicolumn{1}{|c|}{1+a}&\\
\hhline{--~}\multicolumn{1}{|c|}{3+a}&&\\
\hhline{-~~}
\end{array}$ then ${\rm seq}(T)=(3+a, 4+a, 1+a, 4+a, 1+a, 4+a, 2+a, 1+a)$. If
$$T_1=\begin{array}{cc}
\hline\multicolumn{1}{|c|}{7+a}&\multicolumn{1}{|c|}{-4+a}\\
\hline \multicolumn{1}{|c|}{-8+a}&\\
\hhline{-~}
\end{array},\hspace{10pt}
T_2=\begin{array}{cc}
\hline\multicolumn{1}{|c|}{-4+b}&\multicolumn{1}{|c|}{-6+b}\\
\hline \multicolumn{1}{|c|}{-5+b}&\\
\hhline{-~}
\end{array}
$$ with $a-b\notin\mathbb Z$ then ${\rm seq}(T_1, T_2)=(-8+a, 7+a, -4+a, -5+b, -4+b, -6+b).$\end{Ex}
\subsection{Coherent local systems and their annihilators}\label{SScls}

We now recall some results on c.l.s.. 
The defininition of c.l.s. is given in Section~\ref{Sback}. 
In the current section we write $Q=\{Q_n\}$ for a c.l.s., where $Q_n=\{[L_n^\alpha]\}$ for some simple finite-dimensional modules $L_n^\alpha$. 
Since each $L_n^\alpha$ is determined by its dominant $\frak b_n$-highest weight $\lambda_n^\alpha$, we can write $\{\lambda_n^\alpha\}$ instead. 
It is convenient to think of the highest weights $\lambda_n^\alpha$ as functions $f^\alpha\in\mathbb Z_{\ge0}^n\subset\mathbb F^n$ with the normalizing conditions
$$f^\alpha(1)\ge f^\alpha(2)\ge...\ge f^\alpha(n)=0$$
or, equivalently, as partitions with at most $n-1$ parts.  
In this notation, $Q_n=\{f^\alpha\}$. 

The {\it annihilator} $I(Q)$ of a c.l.s. $Q$ is the ideal $\cup_n(\cap_\alpha{\rm Ann}_{{\rm U(\frak{sl}(n))}}L_n^\alpha)\subset{\rm U}(\frak{sl}(\infty))$.

Define functions $f_{k, n}\in\mathbb Z_{\ge0}^n$ by setting
$$f_{k, n}(i):=\begin{cases}1&{\rm~if~}i\le k\\0&{\rm otherwise}\end{cases}.$$

The set of c.l.s. is partially ordered and forms a lattice:

$$Q\subset Q'=\{Q_n\subset Q_n'\},\hspace{10pt}Q\cap Q'=\{Q_n\cap Q'_n\}.$$
In addition, Zhilinskii defines the following Cartan product on c.l.s.:
$$(Q'Q'')_n:=\{f\in\mathbb Z^{n}\mid f=f'+f''{\rm~for~some~}f'\in (Q')_n, f''\in(Q'')_n\},$$ 
see~\cite[Subsection~2.1]{Zh1}. 
A main result of Zhilinskii is that any irreducible c.l.s. is a Cartan product of basic c.l.s.. 
The latter are denoted by $\mathcal L_i, \mathcal R_i, \mathcal L_i^\infty, \mathcal R_i^\infty, \mathcal E, \mathcal E^\infty$, and are defined as follows:
\begin{center}$\mathcal E^\infty$ is the c.l.s. consisting of all integral dominant weights on all levels,\end{center}
$$(\mathcal L_i)_n:=\{f_{k, n}\}_{0\le k\le i},\hspace{10pt}(\mathcal L_i^{\infty})_n:=\{f\in(\mathcal E^\infty)_n\mid f(k)=0{\rm~for~}k> i\},$$
$$(\mathcal R_i)_n:=\{f_{k, n}\}_{n-i\le k\le n},\hspace{10pt}(\mathcal R _i^{\infty})_n:=\{f\in(\mathcal E^\infty)_n\mid f(k)=f(n-i){\rm~for~}k\le n-i\},$$
$$\mathcal E_n:=\{f_{k, n}\}_{0\le k< n}.$$
We can now state 
\begin{proposition} [{Unique factorization property~\cite[Theorem~2.3.1]{Zh1}}] Any proper irreducible c.l.s., i.e., any irreducible c.l.s. non equal $\mathcal E^\infty$, can be expressed uniquely as a Cartan product in the following form:
\begin{equation}{\rm cls}(r', r'', g, {X}, {Y}):=(\mathcal L_{r'}^\infty\mathcal
L_{r'+1}^{l_1-l_2}\mathcal L_{r'+2}^{l_2-l_3}...\mathcal
L_{r'+s}^{l_s-0})~~\mathcal E^g~~(\mathcal R_{r''}^\infty\mathcal
R_{r''+1}^{r_1-r_2}\mathcal R_{r''+2}^{r_2-r_3}...\mathcal R_{r''+t}^{r_t-0})\label{Eclsm}\end{equation}
where $r', r'', g$ are nonnegrative integers, and $${X}=(l_1\ge...\ge l_s>0),\hspace{10pt}{Y}=(r_1\ge...\ge r_t>0)$$ are Young diagrams.
Here, for $r'=0$, $\mathcal L_{r'}^\infty$ is assumed to be the  c.l.s. $\mathcal T$ consisting of the one-dimensional $\frak{sl}(n)$-module at all levels, and similarly $\mathcal R_{r''}^\infty$ is assumed to equal $\mathcal T$ for $r''=0$.
\end{proposition}
As we have shown in~\cite{PP1}, the annihilator $I({\rm cls}(r', r'', g, X, Y))$ depends on the following four parameters
\begin{equation}r:=r'+r'', g, X, Y,\label{Equad}\end{equation}
and all such annihilators are in a natural bijection with quadruples~(\ref{Equad}) where $r, g\in\mathbb Z_{\ge0}$, and $X, Y$ are arbitrary Young diagrams. 
We set $$I(r, g, X, Y):=I({\rm cls}(r, 0, g, X, Y)).$$
It follows from~\cite{Zh1} that $I(r, g, X, Y)$ is a primitive ideal of ${\rm U}(\frak{sl}(\infty))$. 
The main result of~\cite{PP5} claims that the ideals $I(r, g, X, Y)$ exhaust all nonzero proper primitive ideals of ${\rm U}(\frak{sl}(\infty))$. 

Next, following Zhilinskii, we attach to any basic c.l.s. $Q$ a sequence $\gamma(Q; \cdot)$ of $\frak{sl}(2n)$-modules by dispaying the respective highest weights:
$$\gamma(\mathcal L_i; n):=f_{i, 2n},~2n> i,\hspace{10pt}\gamma(\mathcal L_i^{\infty}; n):=(2i-1)f_{i, 2n},~2n>i,$$
$$\gamma(\mathcal R_i; n):=f_{2n-i, 2n},~2n>i,\hspace{10pt}\gamma(\mathcal R_i^{\infty}; n):=(2i-1)f_{2n-i, 2n},~2n>i,$$
$$\gamma(\mathcal E; n):=f_{n, 2n},~n>0.$$
Using~(\ref{Eclsm}) and the rule $\gamma(Q'Q''; n):=\gamma(Q'; n)+\gamma(Q''; n)$, we extend the definition of $\gamma(Q; n)$ to all proper irreducible c.l.s. $Q$. 

To state the next lemma we need to define a {\it precoherent local system} ({\it p.l.s.} for short) $Q$. 
That consists of nonempty sets $Q_n$ of isomorphism classes $[L_n^\alpha]$ of simple finite-dimensional $\frak{sl}(n)$-modules $L_n^\alpha$ for each $n\ge2$, such that each $\frak{sl}(n)$-module $L_n^{\alpha_0}$ branches over $\frak{sl}(n-1)$ as a sum of $\frak{sl}(n-1)$-modules among $L_{n-1}^{\alpha}$. 
For a p.l.s. we do not require that  every $L_{n-1}^\alpha$ with $[L_{n-1}^\alpha]\in Q_{n-1}$ appear in the $\frak{sl}(n-1)$-decomposition of a suitable $L_n^{\alpha_0}$ with $[L_n^{\alpha_0}]\in Q_n$. 

\begin{lemma}\label{Lqch}Let $Q'=\{Q'_n\}$ be a p.l.s. and let $Q$ be an irreducible c.l.s. such that $\gamma(Q; n)\in (Q')_{2n}$ for $n\gg0$. Then $Q\subset Q'$.\end{lemma}
\begin{proof}The statement is implied by~\cite[Lemma~2.3.2]{Zh1}.\end{proof}

\subsection{Highest weight $\frak{sl}(\infty)$-modules and their annihilators}\label{SShwuan}
In what follows we identify $\frak{sl}(\infty)$ with the Lie algebra of traceless matrices $(a_{ij})_{i,j\in\Theta}$ such that each matrix has finitely many nonzero entries. 
We fix the splitting Cartan subalgebra $\frak h$ of diagonal matrices (a detailed discussion of Cartan subalgebras of $\frak{sl}(\infty)$ see in~\cite{DPS}).  
Any subset $S$ of $\Theta$ defines a subalgebra $\frak{sl}(S)$ spanned by
$$\{e_{ij}\}_{i\ne j\in S},\hspace{10pt}\{e_{ii}-e_{jj}\}_{i, j\in S}.$$
If $S$ is infinite then $\frak{sl}(S)\cong\frak{sl}(\infty)$, and $\frak{sl}(S)\cong\frak{sl}(|S|)$ if $S$ is finite, where $|S|$ is the cardinality of $S$.

A total order $\prec$ on $S$ defines a splitting Borel subalgebra
$$\frak b^S(\prec):={\rm span}\{e_{ii}-e_{jj}\}_{i, j\in S}+{\rm span}\{e_{ij}\}_{i\prec j\in S},$$
of $\frak{sl}(S)$, see~\cite{PP2} for more details.

A function $f: S\to\mathbb F$ defines a character $\lambda^S_f$ of $\frak b^S(\prec)$ such that 
$$\lambda^S_f(e_{ij})=0{\rm~for}~i\prec j,\hspace{10pt}\lambda^S_f(e_{ii}-e_{jj})=f(i)-f(j).$$
Let $\mathbb F^S_f$ be the respective one-dimensional $\frak b^S(\prec)$-module and let
$$M^S_\prec(f):=M^S_\prec(\lambda_f):={\rm U}(\frak{sl}(S))\otimes_{{\rm U}(\frak b^S(\prec))}\mathbb F^S_f.$$
Denote by $L^S_\prec(f):=L^S_\prec(\lambda_f)$ the unique simple quotient of $M^S_\prec(\lambda_f)=M^S_\prec(f)$. Put $$I_\prec^S(f):={\rm Ann}_{{\rm U}(\frak{sl}(S))}L^S_\prec(f).$$

If $F\subset\Theta$ is a finite subset, then $L_\prec^F(f)$ is the $\frak{sl}(F)$-module $L(f|_F)$ where the totally ordered set $(F, \prec)$ is naturally identified with $(\{1,..., n\}, <)$. 
In what follows, when given a total order $\prec$ on $\Theta$ and a function $f: \Theta\to\mathbb F$, we will use the above notations $M^\Theta_\prec(f), L^\Theta_\prec(f), I^\Theta_\prec(f)$ having in mind that $\prec$ defines an order on $S$ and $f$ defines a function on $S$ via restriction. 
Whenever $S=\Theta$ we write simply $\frak b(\prec), M_\prec(f), L_\prec(f), I_\prec(f)$. 
Note also that in Section~\ref{Sback} our notation $L_\frak b(\lambda)$ for a simple highest weight module displayed explicitly the relevant Borel subalgebra $\frak b$ and the highest weight $\lambda$, so $L_\prec(f)$ is another notation for $L_{\frak b(\prec)}(\lambda_f)$.

We will be particularly interested in several special kinds of splitting Borel subalgebras. 
\begin{definition} We say that $\frak b^S(\prec)$ is a {\it Dynkin} Borel subalgebra if $(S, \prec)$ is isomorphic as an ordered set to $(\mathbb Z_{>0}, <), (\mathbb Z_{<0}, <)$ or $(\mathbb Z, <)$. 
This is equivalent to the condition that every root of $\frak b^S(\prec)$ is a finite sum of simple roots, see~\cite{Th}.
\end{definition}
Let $\Theta_1, \Theta_2\subset\Theta$ be two subsets. 
We write $\Theta_1\prec \Theta_2$ if $s_1\prec s_2$ for any $s_1\in \Theta_1$ and $s_2\in \Theta_2$.
\begin{definition} We say that $\frak b^S(\prec)$ is an {\it ideal} Borel subalgebra if $S$ can be partitioned into subsets $$S_1\prec S_2\prec S_3$$ such that $(S_1, \prec)\cong (\mathbb Z_{> 0}, <)$ and $(S_3, \prec)\cong (\mathbb Z_{<0}, <)$.\end{definition}
\begin{definition} Let $S\subset\Theta$ be a subset. We say that $f\in\mathbb F^S$ is {\it $\prec$-locally constant on $S$} if there exists a partition $S_1\prec...\prec S_t$ of $S$ such that $f|_{S_i}$ is constant for every $S_i, 1\le i\le t$. 
We say that $f\in\mathbb F^S$ is {\it almost integral on $S$} if there exists a finite set $F \subset S$ such that $f(i)-f(j)\in\mathbb Z$ for all $i, j\in S\backslash F$.\end{definition}

\begin{theorem}[{\cite[Theorem~9]{PP2}}] The following conditions are equivalent:

1) $I_\prec(f)\ne0$,

2) $f$ is $\prec$-locally constant and almost integral on $\Theta$.\end{theorem}

The next proposition relates the computation of the annihilators of simple highest weight $\frak{sl}(\infty)$-modules to the computation of the annihilators of simple highest weight $\frak{sl}(n)$-modules for finite~$n$. 

\begin{proposition}[{\cite[Lemma~5.7]{PP2}}]\label{Pfinf} Let $F_1\subset F_2\subset...\subset F_n\subset...$ be an infinite sequence of finite subsets of $\Theta$, and let $S:=\cup_i F_i$. Then
$$I^S_\prec(f)=\cup_n(\cap_{i\ge n}I^{F_i}_\prec(f)).$$\end{proposition}
\begin{corollary} Let $F_1\subset F_2\subset...\subset F_n\subset...$ be an infinite sequence of finite sets such that $S=\cup_i F_i$. Let $f'\in\mathbb F^S$ be a function such that one of the following holds:

--- $I^{F_i}_{\prec}(f)=I^{F_i}_{\prec}(f')$ for all $i\in\mathbb Z_{>0}$,

--- $I^{F_i}_{\prec}(f)=I^{F_i}_{\prec}(f')$ for all but finitely many $i\in\mathbb Z_{>0}$.\\
Then $I^S_\prec(f)=I^S_\prec(f')$.
\end{corollary}
\begin{corollary}\label{Cknu63} Let $F_1\subset F_2\subset...\subset F_n\subset...$ be an infinite sequence of finite sets such that $S=\cup_i F_i$. Let $f'\in\mathbb F^S$ be a function such that one of the following holds:

--- $f|_{F_i}$ and $f'|_{F_i}$ are connected by a series of shifted admissible interchanges for all $i\in\mathbb Z_{>0}$,

--- $f|_{F_i}$ and $f'|_{F_i}$ are connected by a series of shifted admissible interchanges for all but finitely many $i\in\mathbb Z_{>0}$.
\\
Then $I^S_\prec(f)=I^S_\prec(f')$.
\end{corollary}
\begin{corollary}\label{Crper} Let $\Theta_1\sqcup \Theta_2\sqcup...\sqcup \Theta_t$ be a partition of $\Theta$ and let $f\in\mathbb F^\Theta$ be a function such that $f|_{\Theta_i}$ is constant. Assume that $\prec_1, \prec_2$ are total orders on $\Theta$ such that
\begin{equation}\Theta_1\prec_1 \Theta_2\prec_1...\prec_1 \Theta_t{\rm~and~}\Theta_1\prec_2 \Theta_2\prec_2...\prec_2 \Theta_t.\label{Esucc}\end{equation}
Then $I_{\prec_1}(f)=I_{\prec_2}(f)$.\end{corollary}
\begin{definition} Let $\prec_1,\prec_2$ be total orders on $\Theta$, and $f\in\mathbb F^\Theta$ be a function. 
We say that $\prec_1$ and $\prec_2$ are {\it $f$-equivalent} if $f$ is locally constant with respect to a partition $\Theta_1\sqcup ...\sqcup \Theta_t$ of $\Theta$ and this partition satisfies~(\ref{Esucc}).\end{definition}
Corollary~\ref{Crper} claims that $I_{\prec_1}(f)=I_{\prec_2}(f)$ for $f$-equivalent total orders $\prec_1, \prec_2$. 
\begin{definition} Let $f\in\mathbb F^{\Theta}$ be a function, almost integral and locally constant, and let $\Theta_1\prec...\prec \Theta_t$ be some partition of $\Theta$. 
We say that this partition is {\it $f$-preferred} if 

\begin{center}$(\Theta_1, \prec)\cong(\mathbb Z_{>0}, <)$, $(\Theta_t, \prec)\cong(\mathbb Z_{<0}, <)$, $(\Theta_i, \prec)\cong(\mathbb Z, <)$ for $1<i<t$,\end{center}
and for all $i$ there exist $s_i^-\in \Theta_i$ and $s_{i+1}^+\in \Theta_{i+1}$ such that $f(s)=f(s')$ for all $s\in \Theta_i, s_i^-\prec s,$ and $s'\in \Theta_{i+1}, s'\prec s_{i+1}^+$. 
We say that a total order $\prec_f$ is {\it $f$-preferred} if there exists a partition $\Theta_1\prec_f...\prec_f \Theta_t$ which is $f$-preferred with respect to $\prec_f$.\end{definition}
Let $f\in\mathbb F^{\Theta}$ be an almost integral and locally finite function with respect to a partition $\Theta_1\prec...\prec \Theta_t$ of $\Theta$. 
It is easy to construct an $f$-preferred order $\prec_f$ on $\Theta$ such that $\prec_f$ is $f$-equivalent to $\prec$. 
Indeed, let $i_1, i_2,..., i_q$ be the set of indices such that $\Theta_{i_1},..., \Theta_{i_q}$ are infinite. 
We split each ordered set $\Theta_{i_k}$ into two infinite sets $\Theta_{i_k}^l, \Theta_{i_k}^r$ so that
$$\Theta_{i_1}^l\prec \Theta_{i_1}^r\prec...\prec \Theta_{i_{q}}^l\prec \Theta_{i_{q}}^r.$$
As a result, $\Theta$ equals the disjoint union
\begin{equation}(\Theta_1\sqcup \Theta_2\sqcup...\sqcup \Theta_{i_1}^l)\sqcup(\Theta_{i_1}^r\sqcup \Theta_{i_1+1}\sqcup...\sqcup \Theta_{i_2}^l)\sqcup...\sqcup (\Theta_{i_q}^r\sqcup \Theta_{i_q+1}\sqcup...\sqcup \Theta_t),\label{Eorddec}\end{equation}
and we have
$$\Theta_1\prec \Theta_2\prec...\prec \Theta_{i_1}^l\prec \Theta_{i_1}^r\prec \Theta_{i_1+1}\prec... \prec \Theta_{i_2}^l \prec...\prec \Theta_{i_q}^r\prec \Theta_{i_q+1}\prec...\prec \Theta_t.$$
The desired order $\prec_f$ will be $f$-preferred with respect to the decomposition~(\ref{Eorddec}).
To inroduce $\prec_F$, we start with $(\Theta_1\sqcup \Theta_2\sqcup...\sqcup \Theta_{i_1}^l)$ and replace the given order $\prec$ by an order $\prec_f$ isomorphic to $(\mathbb Z_{>0}, >)$ such that $\Theta_1\prec_f \Theta_2\prec_f...\prec_f \Theta_{i_1}^l$. 
Next, for $(\Theta_{i_1}^r\sqcup \Theta_2\sqcup...\sqcup \Theta_{i_1}^l)$ we replace the given order $\prec$ by an order $\prec_f$ isomorphic to $(\mathbb Z, >)$ such that $\Theta_{i_1}^r\prec_f \Theta_{i_1+1}\prec_f...\prec_f \Theta_{i_2}^l$. 
We repeat this last step $q-2$ times. 
Finally, at the right end $(\Theta_{i_q}^r\sqcup \Theta_2\sqcup...\sqcup \Theta_{i_t})$ we replace the order $\prec$ by an order $\prec_f$ isomorphic to $(\mathbb Z_{<0}, <)$ such that $\Theta_{i_q}^r\prec_f \Theta_{i_q+1}\prec_f...\prec_f \Theta_{i_t}$. 
The so obtained order $\prec_f$ is $f$-preferred and is $f$-equivalent to the original order. 
Therefore, Corollary~\ref{Crper} implies $$I_{\prec}(f)=I_{\prec_f}(f).$$ 
\section{Robinson-Schensted algorithm at infinity}

In what follows we extend the RS-algorithm to stably decreasing infinite sequences. 
Overall, the procedure is very similar to (and is based on) the one given in Subsection~\ref{SSrskf}. 
We consider functions $f\in\mathbb F^{\mathbb Z_{>0}}, \mathbb F^{\mathbb Z_{<0}}, \mathbb F^{\mathbb Z}$ and identify them with the respective sequences  
$$f(1), f(2),...,$$
$$..., f(-2), f(-1)$$
$$ ..., f(-1), f(0), f(1),...$$
Admissible interchanges for functions $f\in \mathbb F^{\mathbb Z_{>0}}, \mathbb F^{\mathbb Z_{<0}}, \mathbb F^{\mathbb Z}$ are defined as in Subsection~\ref{SSkm}.
\begin{definition}We say that $f$ is {\it stably decreasing} if $f(i)>_\mathbb Zf(i+1)$ for $|i|\gg0$.\end{definition}
The formula $\rho(i)=-i$ defines three different functions $\rho_{\mathbb Z_{<0}}\in\mathbb F^{\mathbb Z_{<0}}, \rho_{\mathbb Z}\in\mathbb F^{\mathbb Z}, \rho_{\mathbb Z_{>0}}\in\mathbb F^{\mathbb Z_{>0}}.$ 

The following ``insertion operation'' inserts a given $f_1\in\mathbb F^s$ into positions $i_1<n_2<...<i_r$ of $f_2$:

$${\rm ins}(i_1,..., i_r; f_1, f_2)(i):=
\begin{cases}f_2(i)&{\rm~if~}i< i_1,
\\f_1(t)&{\rm~if~}i=i_t, 1\le t\le r,
\\f_2(i-t)&{\rm~if~}i_t<i<i_{t+1}, 1\le t\le r-1
\\f_2(i-s)&{\rm~if~}i>i_r\end{cases},\hspace{10pt}{\rm~for~}f\in\mathbb F^{\mathbb Z},\mbox{~or~}f\in\mathbb F^{\mathbb Z_{>0}}\mbox{~and~}i_1\ge0,$$
$${\rm ins}(i_1,..., i_r; f_1, f_2)(i):=
\begin{cases} f_2(i+s)&{\rm~if~}i<i_1
\\f_1(t)&{\rm~if~}i=i_t, 1\le t\le r,
\\f_2(i+s-t)&{\rm~if~}i_t<i<i_{t+1}, 1\le t\le r-1
\\f_2(i)&{\rm~if~}i_r<i\end{cases},\hspace{10pt}{\rm~for~}f\in\mathbb F^{\mathbb Z_{<0}}, i_1\le0.$$
\begin{remark} Since the (shifted) admissible interchanges  and the respective equivalence classes are defined for all sequences regardless of any stabilization conditions, it could be an interesting combinatorial problem to study the corresponding equivalence classes. \end{remark}

\subsection{{\bf Left-infinite case ($\mathbb Z_{<0}$)}} Consider a stably decreasing sequence $f\in\mathbb F^{\mathbb Z_{<0}}$.  
We now explain how to apply the infinite RS-algorithm to $f$. 
What we do is simply apply the RS-algorithm consecutively to the finite tails $f(-n),..., f(0)$ of $f$.  
Then, for $n\ll0$, the RS-algorithm will keep modifying only one of the tableaux in the outputs of previous steps. 
This follows from the fact for $n\ll0$ the numbers $f(n)$ are in same integrality class.

Next, note that, since $f$ is stably decreasing, this modification will amount to adding the box $\begin{tabular}{|c|}\hline $f(-n-1)$\\\hline\end{tabular}$ to the left-hand side of the first row. 
In this way, the output $RS(f)$ of our algorithm consists of several (possibly none) finite Young tableaux and one tableau whose first row is infinite and all other rows are finite. 
Denote the infinite tableau by $T_1$ and the other tableaux by $T_2,..., T_s$. Denote the first row of $T_1$ by ${\rm\overline{seq}}(f)$, $T_1$ without the first row by $T_1'$. Set ${\rm \underline{seq}}(f):={\rm seq}(T_1', T_2, T_3,..., T_s)$. 
Then it is straightforward to check that\begin{equation}RS(f)=RS({\rm ins}(i_1,..., i_r; {\rm \underline{seq}}(f), {\rm \overline{seq}}(f)),\label{Erslc}\end{equation}
where $r$ is the number of elements in ${\rm \underline{seq}}(f)$ and $i_1, i_2, ..., i_r,$ are integers such that
\begin{equation}i_{k+1}>i_k+1, {\rm \overline{seq}}(f)_{i_r}>_{\mathbb Z}{\rm \underline{seq}}(f)_{k}{\rm~or~}{\rm \overline{seq}}(f)_{i_r}-{\rm \underline{seq}}(f)_{k}\notin\mathbb Z{\rm~for~all~}k\le r\label{Ccondi}\end{equation}
(the condition on $i_r$ is satisfied for $i_r\ll0$). 
The equality~(\ref{Erslc}) plays an important role in our main result below. 
\subsection{{\bf Right-infinite case ($\mathbb Z_{>0}=-\mathbb Z_{<0}$)}} Let $f\in\mathbb F^{\mathbb Z_{>0}}$ be a stably decreasing function. It is clear that the sequence
$$f^*:=(..., -f(3), -f(2), -f(1))$$
is an element of $\mathbb F^{\mathbb Z_{<0}}$ and is stably decreasing. If $g\in\mathbb F^{\mathbb Z_{<0}}$, we set $g^*$ to be the sequence
$$-g(0), -g(-1), -g(-2),...$$
Then $(f^*)^*=f$ for $f\in\mathbb F^{\mathbb Z_{>0}}$ or $\mathbb F^{\mathbb Z_{<0}}$. 

In this case, we have $$RS(f^*)=RS({\rm ins}(i_1,..., i_s; {\rm \underline{seq}}(f^*)^*,~{\rm \overline{seq}}(f^*)^*)^*),$$
where $s$ is the number of elements in ${\rm \underline{seq}}(f^*)$ and $i_1,..., i_r$ satisfy the mirror image of~(\ref{Ccondi})
$$i_{k+1}>i_k+1, {\rm \overline{seq}}(f^*)_{-i_r}>_{\mathbb Z}{\rm \underline{seq}}(f^*)_{k}{\rm~or~}{\rm \overline{seq}}(f^*)_{-i_r}-{\rm \underline{seq}}(f^*)_{k}\notin\mathbb Z{\rm~for~all~}k\le r.$$
\begin{remark} In the procedure presented in this subsection, we apply the RS-algorithm inductively starting from the ``infinite tail'' of our sequence $f$. It also makes sense to apply the RS-algorithm starting from the beginning of the sequence $f$. The result will differ by an analogue of the Schutzenberger involution, see~\cite{Knu}.\end{remark}
\subsection{{\bf Two-sided case ($\mathbb Z$)}} Consider a stably decreasing almost integral sequence $f\in\mathbb F^{\mathbb Z}$. 
We say that $f$ is {\it almost integral} if $f(n_+)-f(n_-)\in\mathbb Z$ for $n_-\ll0$ and $n_+\gg0$. 

Assume $f$ is almost integral. 
To apply the infinite RS-algorithm to $f$, all we have to do is to apply the RS-algorithm to ``middle'' finite subsequences $f(n_-),..., f(n_+)$ of $f$ where $n_-\to-\infty$ and $n_+\to+\infty$.  
Note that for $n_-\ll0$ and $n_+\gg0$ the RS-algorithm will keep modifying only one of the tableaux in the outputs of previous steps. 
This follows from the fact for $n_-\ll0$ and $n_+\gg0$ the numbers $f(n_-), f(n_+)$ are in same integrality class.

Next, note that since $f$ is stably decreasing this modification will amount to adding the boxes $\begin{tabular}{|c|}\hline $f({n_{-}}-1)$\\\hline\end{tabular}$ to the left-hand side or $\begin{tabular}{|c|}\hline $f(n_{+}+1)$\\\hline\end{tabular}$ to the right-hand side of the first row. 
In this way, the output $RS(f)$ of our algorithm consists of several finite Young tableaux (possibly none) and one tableau whose first row is infinite and all other rows are finite. 
Denote the infinite tableau by $T_1$ and the other tableaux by $T_2,..., T_s$. Denote the first row of $T_1$ by ${\rm\overline{seq}}(f)$, $T_1$ without the first row by $T_1'$.  
Note that the identification of two-sided sequences with $\mathbb F^\mathbb Z$ is unique only up to a shift, and we fix this shift in such a way that $f(i)={\rm\overline{seq}}(f)_i$ for $i\ll0$.

Set ${\rm \underline{seq}}(f):={\rm seq}(T_1', T_2, T_3,..., T_s)$.  
Then we point out that~(\ref{Erslc}) holds also in this case where $r$ is the number of elements in ${\rm \underline{seq}}(f)$ and $i_1, i_2, ..., i_r$ satisfy~(\ref{Ccondi}).

\subsection{Admissible interchanges and Robinson-Schensted algorithm at infinity}
The next proposition is an infinite-dimensional version of the equivalence of claims 1) and 3) in Theorem~\ref{Tjo}.
\begin{proposition}For stably decreasing functions $f, f'$ from $\mathbb F^{\mathbb Z_{>0}}, \mathbb F^{\mathbb Z_{<0}}$ or $\mathbb F^{\mathbb Z}$, the following conditions are equivalent:

a) $f$ and $f'$ are connected by a series of admissible interchanges,

b) $RS(f)=RS(f')$.\end{proposition}
\begin{proof}Elementary and straightforward.\end{proof}
\section{Two attributes of an ideal in ${\rm U}(\frak{sl}(\infty))$}
In this section, we introduce a sequence of algebraic varieties associated with an ideal $I\subset{\rm U}(\frak{sl}(\infty))$, as well as a c.l.s. associated with $I$. 

Recall that ${\rm Z}(\frak{sl}(n))$ stands for the centre of ${\rm U}(\frak{sl}(n))$. 
Denote by ${\rm Irr}_n$ the set of isomorphism classes of simple finite-dimensional ${\rm U}(\frak{sl}(n))$-modules. 
\begin{lemma}[{cf.~\cite[Subsection~3.1]{BJ}}]\label{Lprebor} Let $I_1, I_2$ be ideals of ${\rm U}(\frak{sl}(n))$. 
Then the following conditions are equivalent:

a) $I_1+I_2={\rm U}(\frak{sl}(n))$,

b) $1\in I_1+I_2$,

c) $(I_1\cap {\rm Z}(\frak{sl}(n)))+(I_2\cap {\rm Z}(\frak{sl}(n)))={\rm Z}(\frak{sl}(n))$,

d) $1\in (I_1\cap {\rm Z}(\frak{sl}(n)))+(I_2\cap {\rm Z}(\frak{sl}(n)))$.\end{lemma}
\begin{proof}It is clear that a) is equivalent to b), and that c) is equivalent to d). 
Hence it is enough to prove that b) is equivalent to d).

As an $\frak{sl}(n)$-module with the adjoint structure, ${\rm U}(\frak{sl}(n))$ is locally finite, and is an infinite direct sum of $\frak{sl}(n)$-isotypic components ${\rm U}(\frak{sl}(n))_\lambda$ where $\lambda$ runs over the entire set ${\rm Irr}_n$. 
Hence, $$I_1=\oplus_{\lambda\in {\rm Irr}_n} ({\rm U}(\frak{sl}(n))_\lambda\cap I_1),\hspace{10pt}I_2=\oplus_{\lambda\in {\rm Irr}_n} ({\rm U}(\frak{sl}(n))_\lambda\cap I_2),$$
and 
$$((I_1+I_2)\cap {\rm Z}(\frak{sl}(n)))=(I_1+I_2)^\frak g=(I_1)^\frak g+(I_2)^\frak g=({\rm Z}(\frak{sl}(n))\cap I_1)+({\rm Z}(\frak{sl}(n))\cap I_2),$$
where $*^\frak g$ stands for $\frak g$-invariants. 
This implies that b) is equivalent to d).
\end{proof}
\begin{lemma}\label{Lbor}Let $I$ be an ideal of ${\rm U}(\frak{sl}(n))$ and $L$ be a simple finite-dimensional $\frak{sl}(n)$-module. 
If \begin{equation}I\cap {\rm Z}(\frak{sl}(n))\subset ({\rm Z}(\frak{sl}(n))\cap{\rm Ann}_{{\rm U}(\frak{sl}(n))}L),\label{Ebor}\end{equation} then $I\subset {\rm Ann}_{{\rm U}(\frak{sl}(n))}L$.\end{lemma}
\begin{proof}Note that~(\ref{Ebor}) implies $$1\not\in ((I\cap {\rm Z}(\frak{sl}(n)))+({\rm Z}(\frak{sl}(n))\cap{\rm Ann}_{{\rm U}(\frak{sl}(n))}L))={\rm Z}(\frak{sl}(n))\cap{\rm Ann}_{{\rm U}(\frak{sl}(n))}L.$$ 
Moreover, it follows from Lemma~\ref{Lprebor} that $1\notin(I+{\rm Ann}_{{\rm U}(\frak{sl}(n))} L)$.
It is a well-known result that there exists a unique maximal ideal $m$ of ${\rm U}(\frak{sl}(n))$ containing ${\rm Z}(\frak{sl}(n))\cap {\rm Ann}_{{\rm U}(\frak{sl}(n))}L$, see~\cite[Subsection~1.1]{BJ}. 
Clearly, ${\rm Ann}_{{\rm U}(\frak{sl}(n))}L$ is maximal, and hence $$I+{\rm Ann}_{{\rm U}(\frak{sl}(n))}L\subset m={\rm Ann}_{{\rm U}(\frak{sl}(n))}L.$$ 
This implies $I\subset {\rm Ann}_{{\rm U}(\frak{sl}(n))}L$.\end{proof}
Fix $I\subset {\rm U}(\frak{sl}(\infty))$. 
For any $n\ge2$, we set
$$Q_n(I):=\{[L]\in {\rm Irr}_n\mid I\cap {\rm U}(\frak{sl}(n))\subset {\rm Ann}_{{\rm U}(\frak{sl}(n))}L\}.$$
The union of $Q_n(I)$ is a p.l.s., see Subsection~\ref{SScls}. 
Proposition~4.8 of~\cite{PP5} implies that therere exists a c.l.s. $Q(I)$ such that $Q(I)_n=Q_n(I)$ for $n\gg0$. 
Such c.l.s. $Q(I)$ is clearly unique.

A theorem of Harish-Chandra claims that ${\rm Z}(\frak{sl}(n))$ is isomorphic to the ${ S}_n$-invariants ${\rm S}(\frak h_n)^{S_n}$ in the symmetric algebra ${\rm S}(\frak h_n)^{S_n}$.  
Therefore the radical ideals of ${\rm Z}(\frak{sl}(n))$ are in one-to-one correspondence with the ${ S}_n$-invariant subvarieties $\frak h_n^*$. 
Let $f\in\mathbb F^n$ be a function. 
Then the ideal $I(f)\cap {\rm Z}(\frak{sl}(n))$ is maximal; it corresponds to the $S_n$-orbit of the weight $\lambda_f+\rho_n$ where
$$\rho_n:=\lambda_{n, n-1,..., 1}.$$

Let $I$ be an ideal of ${\rm U}(\frak{sl}(\infty))$. 
Consider $I\cap {\rm Z}(\frak{sl}(n))$. Clearly, $I\cap {\rm Z}(\frak{sl}(n))$ is an ideal of ${\rm Z}(\frak{sl}(n))$ and $\sqrt{I\cap {\rm Z}(\frak{sl}(n))}$ is a radical ideal; 
it is identified with the $S_n$-stable subvariety ${\rm ZVar}_n(I)$ of $\frak h_n^*$.

The variety ${\rm ZVar}_n(I)$ and the set $Q_n(I)$ are related as follows.
\begin{proposition}\label{Pzvar} Let $I$ be a primitive ideal of ${\rm U}(\frak{sl}(\infty))$. Then

a) ${\rm ZVar}_n(I)$ equals the Zariski closure of the set $\{w(\lambda_f+\rho_n)\mid [L(f)]\in Q_n(I), w\in S_n\}$.

b) Let $f$ be a dominant function such that $\lambda_f+\rho_n\in {\rm ZVar}_n(I)$. Then $[L(f)]\in Q_n(I)$.\end{proposition}
\begin{proof} If $I$ is primitive then $I$ is locally integrable, see~\cite[Section~4]{PP5}, which means that
$$I\cap {\rm U}(\frak{sl}(n))=\cap_{[L^\alpha]\in Q(I)_n}{\rm Ann}_{{\rm U}(\frak{sl}(n))} L^\alpha.$$
This implies 
$$I\cap {\rm Z}(\frak{sl}(n))=\cap_{[L^\alpha]\in Q(I)_n}({\rm Ann}_{{\rm U}(\frak{sl}(n))} L^\alpha\cap {\rm Z}(\frak{sl}(n)),$$
and a) follows.

We proceed to b).  
The condition $\lambda_f+\rho_n\in {\rm ZVar}_n(I)$ implies $$I\cap{\rm Z}(\frak{sl}(n))\subset ({\rm Z}(\frak{sl}(n))\cap{\rm Ann}_{{\rm U}(\frak{sl}(n))}L(f)).$$
To finish the proof we use Lemma~\ref{Lbor}.
\end{proof}

Consider $f\in\mathbb F^{\Theta}$ together with an arbitrary total order $\prec$ on $\Theta$. 
Put
$${\rm F}_{\prec, n}(f):=\{(f(i_1)-1,..., f(i_n)-n)\in\mathbb F^n\mid i_1\prec ...\prec  i_n\in\Theta\}.$$
\begin{lemma}\label{Ltop} We have
$$\lambda_g\in{\rm ZVar}_n(I_\prec(f))$$
for all $g\in{\rm F}_{\prec, n}(f)$.
\end{lemma}
\begin{proof} Consider a finite subset $F=\{i_1,..., i_n\}$ of $\Theta$. 
It is clear that $L_\prec^F(f)$ is an $\frak{sl}(F)$-subquotient of $L_\prec(f)$. 
This implies that $$I_\prec(f)\cap{\rm U}(\frak{sl}(n))\subset I_\prec(f),$$
and hence that
$$I_\prec(f)\cap{\rm Z}(\frak{sl}(n))\subset I_\prec(f)\cap{\rm Z}(\frak{sl}(n)).$$
The latter inclusion is equivalent to the desired statement.
\end{proof}
\begin{corollary}\label{Clam} Assume that $f$ and $f'$ are connected by a series of admissible interchanges. 
Then
$$\lambda_g\in{\rm ZVar}_n(I_\prec(f))$$
for all $g\in{\rm F}_{\prec, n}(f')$.
\end{corollary}

\section{The main result for Dynkin Borel subalgebras}\label{SDynk}
Assume that $\frak b(\prec)$ is a Dynkin Borel subalgebra. 
This means that we can identify the ordered set $(\Theta, \prec)$ with one of the three ordered sets $(\mathbb Z_{>0}, <), (\mathbb Z_{<0}, <)$, $(\mathbb Z, <)$. 

Let $f\in\mathbb F^{\mathbb Z_{<0}}$, $\mathbb F^{\mathbb Z_{>0}}$, $\mathbb F^{\mathbb Z}$ be a locally constant function. 
Clearly, this is equivalent to
\begin{equation}\exists N\in\mathbb Z_{>0}: f(i)=f(i+1){\rm~for~all~}|i|\ge N.\label{Estc}\end{equation}
We fix such an $N$. Put
$$h^\pm(f):=\lim\limits_{n\to\pm\infty}f(n)$$
cf.~(\ref{Estc}). Note that if $(\Theta, \prec)\cong (\mathbb Z_{<0}, <)$ or $(\Theta, \prec)\cong (\mathbb Z_{>0}, <)$ then any locally constant function $f\in\mathbb F^\Theta$ is almost integral. 
If $f\in\mathbb F^\mathbb Z$ is almost integral and locally constant then $h^+(f)-h^-(f)$ is an integer. 

For a locally constant function $f\in\mathbb F^{\mathbb Z_{<0}}$, we set 
$$f^+:=(..., f(i)+i,..., f(-2)+2, f(-1)+1).$$
Then $f^+$ is a stably decreasing function in $\mathbb F^\Theta$. 
It is easy to see that $${\rm\overline{seq}}(f^+)_i=h^-(f)+|{\rm\underline{seq}}(f^+)|-i{\rm~for~}i\le -N$$
or, equivalently, $${\rm\overline{seq}}(f^+)^*_i=-h^-(f)-|{\rm\underline{seq}}(f^+)|-i{\rm~for~}i\ge N.$$
Hence the function\begin{equation}{\rm\overline{seq}}(f^+)^*+h^-(f)+|{\rm\underline{seq}}(f^+)|-\rho_{\mathbb Z_{>0}}\label{Eseqr},\end{equation}
where $h(f), |{\rm\underline{seq}}(f^+)|$ denote the constant functions, is nonincreasing and is stably equal to zero. 
The nonzero values of the function~(\ref{Eseqr}) form a partition which we denote by ${Y}(f)$.
\begin{proposition}\label{Prsz-}Let $f\in\mathbb F^{\mathbb Z_{<0}}$ be a locally constant function. 
Then $I_<(f)=I(r(f), 0, \emptyset, {Y}(f))$ where $r:=r(f):=|{\rm\underline{seq}}(f^+)|$.
\end{proposition}
\begin{proof} It is enough to prove

a) $I(r(f), 0, \emptyset, {Y}(f))\subset I_<(f)$,\\
and

b) $I_<(f)\subset I(r(f), 0, \emptyset, {Y}(f))$.

Statement a) is equivalent to
\begin{equation}I(r(f), 0, \emptyset, {Y}(f))\subset I_<(f')\label{Est1}\end{equation}
for some $f'$ such that $f$ and $f'$ are connected by a series of admissible interchanges. We pick $f'$ as in~(\ref{Erslc}) with $i_1,..., i_r$ satisfying~(\ref{Ccondi}). Then we apply~\cite[Lemma~5.4]{PP2} and Proposition~\ref{Pfinf} to the inserted variables. This shows a).

Theorem 3.2 of~\cite{PP2} implies that b) is equivalent to

b$'$) $Q_n(I(r(f), 0, \emptyset, {Y}(f)))\subset Q_n(I_<(f))$ for $n\gg0$.

According to~\cite[Lemma~7.6c)]{PP1} we have
$$Q_n(I(r(f), 0, \emptyset, {Y}(f)))=\cup_{r'+r''=r}{\rm cls}(r', r'', \emptyset, {Y}(f)).$$
Hence we need to prove that
$${\rm cls}(r', r'', 0, \emptyset, {Y}(f))_n\subset Q_n(I_<(f))$$
for any $n\gg0$ and all nonnegative integers $r', r''$ such that $r'+r''=r$. 

We fix $r', r''$ with $r'+r''=r$. 
Let the partition ${Y}(f)$ be $(l_1\ge...\ge l_s>0)$. 
We also fix $n\ge r+s$. 
Then Lemma~\ref{Ltop} asserts that
\begin{equation}\lambda_{g}\in {\rm ZVar}_n(I_<(f)),\label{Ela}\end{equation}
for any $\lambda_g\in F_{<, n}(f)$. 

We will now make use of Corollary~\ref{Cknu63} which allows us to replace $f$ in the formula~(\ref{Ela}) by any $f'$ which is connected with $f$ by a series of shifted admissible interchanges. 
Let $i_1,..., i_r$ be integers satisfying condition~(\ref{Ccondi}). 
Consider the subset \begin{equation}F_{i_1,..., i_r}:=\{i_1,..., i_r,-(n-r),..., -1\}\subset \Theta.\label{Efth}\end{equation} 
Define $f_{i_1,..., i_r}$ by the requirement that $(f_{i_1,..., i_r})^+$ equals the right-hand side of~(\ref{Erslc}) applied to $f^+, i_1,..., i_r$. 
The order of the elements of $F_{i_1,..., i_r}$ in (\ref{Efth}) allows us to consider $f_{i_1,..., i_r}|_{F_{i_1,..., i_r}}$ as a vector in $\mathbb F^n$. 
Since $f$ and $f_{i_1,..., i_r}$ are connected by a series of shifted admissible interchanges, Corollary~\ref{Clam} implies that
$$\lambda_{g'}\in {\rm ZVar}_n(I_<(f))={\rm ZVar}_n(I_<(f_{i_1,..., i_r})),$$where
$$g'(k)=\begin{cases}{\rm\underline{seq}}(f^+)_k+i_k-k&{\rm if~}1\le k\le r\\
{\rm\overline{seq}}(f^+)_{k-1-n+s}+(k-1-n+s)-k=h^-(f)+r-k&{\rm if~}r<k\le n-s\\
{\rm\overline{seq}}(f^+)_{k-1-n+s}+(k-1-n+s)-k=h^-(f)-r-k-l_{n+1-k}&{\rm if~}n-s<k\le n\\
\end{cases}.$$ 

For all choices of negative integers $i_1,..., i_r$ satisfying~(\ref{Ccondi}), the above weights $\lambda_{g'}$ form a subset of ${\rm ZVar}_n(I_<(f))$ whose Zariski closure contains the set $\lambda_{g''}$ for any $g''$ of the form
$$g''(k)=\begin{cases} i_k&{\rm if~}1\le k\le r\\
h^-(f)+r-k&{\rm if~}r<k\le n-s\\
h^-(f)+r-k-l_{n+1-k}&{\rm if~}n-s<k\le n\\
\end{cases}$$
where now $i_1,..., i_r\in\mathbb F$ are arbitrary. 
Therefore, Proposition~\ref{Pzvar} implies
\begin{equation}[L(i_1,..., i_{r'}, l_1, l_2,..., l_s, 0,..., 0, -j_{r''}, ..., -j_2, -j_1)]\in Q(I_<(f))_n\label{Eij}\end{equation}
for all positive integers $i_1\ge i_2\ge...\ge i_{r'}, j_1\ge...\ge j_{r''}$ such that $i_r\ge l_1, j_{r''}\ge0$. 
Consequently, \begin{equation}\gamma({\rm cls}(r', r'', 0, \emptyset, {Y}(f)); n)\in Q(I_<(f))_{2n}.\label{Est3}\end{equation}
Now Lemma~\ref{Lqch} implies b$'$), and the proof is complete.\end{proof}
\begin{proposition}\label{Prsz+}Let $f\in\mathbb F^{\mathbb Z_{>0}}$ be a locally constant function. 
Then $I_<(f)=I(r(f^*), 0, {Y}(f^*), \emptyset).$
\end{proposition}
\begin{proof} This proposition can be proved by repeating the proof of Proposition~\ref{Prsz-} and making some obvious changes. 
For a shorter proof, note that the outer automorphism 
$$e_{ij}\mapsto-e_{ji}$$
of $\frak{sl}(\infty)$ interchanges the simple modules $L_<(f)$ and $L_>(-f)\cong L_<(f^*)$ ($(\mathbb Z_{>0}, >)$ is isomorphic to $(\mathbb Z_{<0}, <)$ and thus $L_>(-f)\cong L_<(f^*)$), and interchanges the ideals $I(r(f^*), 0, {Y}(f^*), \emptyset)$ and $I(r(f^*), 0, \emptyset, {Y}(f^*))$. 
Therefore the statement also follows from Proposition~\ref{Prsz-}.
\end{proof}
Consider now the case $(\Theta, \prec)=(\mathbb Z, <)$. It is clear that if $f\in\mathbb F^{\mathbb Z}$ is a locally constant function, then
$$f^+:=(..., f(-1)+1, f(0), f(1)-1,..., f(i)-i,...)$$
is an element of $\mathbb F^{\mathbb Z}$ and is stably decreasing.
\begin{proposition}\label{Prsz}Let $f\in\mathbb F^{\mathbb Z}$ be a locally constant function. Let $r:=r(f)=|{\rm\underline{seq}}(f^+)|$.
Then $$I_<(f)=I(r, h^-(f)-h^+(f)+r, \emptyset, \emptyset).$$
\end{proposition}
\begin{proof} The proof follows the same idea as the proof of Proposition~\ref{Prsz-}. 
Below we highlight the necessary changes.  

The inclusion 
$$I(r, h^-(f)-h^+(f)+r, \emptyset, \emptyset)\subset I_<(f)$$
is equivalent to
\begin{equation}I(r, h^-(f)-h^+(f)+r, \emptyset, \emptyset)\subset I_<(f')\label{Eext}\end{equation}
where $f'$ is as in~(\ref{Erslc}). 
By applying~\cite[Lemma~5.4]{PP2} and Proposition~\ref{Pfinf} to the inserted variables we establish~(\ref{Eext}).

Next,  Theorem 3.2 of~\cite{PP2} implies that the inclusion
$$I_<(f)\subset I(r, h^-(f)-h^+(f)+r, \emptyset, \emptyset)$$
is equivalent to the inclusions
$$Q_n(I(r, h^-(f)-h^+(f)+r, \emptyset, \emptyset))\subset Q_n(I_<(f)){\rm~for~}n\gg0.$$

As in the proof of Proposition~\ref{Prsz-}, it suffices to show that
$$\cup_{r'+r''=r}{\rm cls}(r', r'', h^-(f)-h^+(f)+r, \emptyset, \emptyset)\subset Q_n(I_<(f)){\rm~for~}n\gg0.$$

We fix nonnegative integers $r', r''$ with $r'+r''=r$. 
For $n\ge r$ Lemma~\ref{Ltop} asserts that
\begin{equation}\lambda_{g}\in {\rm ZVar}_n(I_<(f)),\label{Ela2}\end{equation}
for any $\lambda_g\in F_{<, n}(f)$. 
We now replace $f$ in formula~(\ref{Ela2}) by an appropriate $f'$ with $I_<(f)=I_<(f')$.
Let $i_1,..., i_r$ be integers satisfying~(\ref{Ccondi}). 
Consider the subset \begin{equation}F_{i_1,..., i_r}:=\{i_1,..., i_r,-(n-r')-N,..., -1-N, 1+N, 2+N,..., (n-r'')+N\}\subset \Theta.\label{Efth2}\end{equation} 
Define $f_{i_1,..., i_r}$ by the requirement that $(f_{i_1,..., i_r})^+$ equals the right-hand side of~(\ref{Erslc}) applied to $f^+, i_1,..., i_r$. 
The order of the elements of $F_{i_1,..., i_r}$ in (\ref{Efth2}) allows us to consider $f_{i_1,..., i_r}|_{F_{i_1,..., i_r}}$ as a vector in $\mathbb F^{2n}$. 
Then $I_<(f)=I_<(f_{i_1,..., i_r})$. 
Moreover, Corollary~\ref{Clam} implies that
$$\lambda_{g'}\in {\rm ZVar}_{2n}(I_<(f))={\rm ZVar}_{2n}(I_<(f_{i_1,..., i_r})),$$where
$$g'(k)=\begin{cases}{\rm\underline{seq}}(f^+)_k+i_k-k&{\rm if~}1\le k\le r\\
{\rm\overline{seq}}(f^+)_{k-1-N-n-r''}+(k-1-N-n-r'')-k=h^-(f)-k&{\rm if~}r<k\le n+r''\\
{\rm\overline{seq}}(f^+)_{k-n-r''+N}+(k-n-r''+N)-k=h^+(f)-k&{\rm if~}n+r''<k\le 2n\\
\end{cases}.$$

For all integers $i_1,..., i_r$ satisfying~(\ref{Ccondi}), the above weights $\lambda_{g'}$ form a subset of ${\rm ZVar}_{2n}(I_<(f))$ whose Zariski closure contains the set $\lambda_{g''}$ for any $g''$ of the form
$$g''(k)=\begin{cases} i_k&{\rm if~}1\le k\le r\\
{\rm\overline{seq}}(f^+)_k-k=h^-(f)-k&{\rm if~}r<k\le n\\
{\rm\overline{seq}}(f^+)_k-k=h^+(f)-k&{\rm if~}n<k\le 2n\\
\end{cases}$$
where now $i_1,..., i_r\in\mathbb F$ are arbitrary. 
Therefore Proposition~\ref{Pzvar} implies
$$[L(i_1,..., i_{r'}, \underbrace{h^-(f),..., h^-(f)}_{(n-r')-{\rm times}}, \underbrace{h^+(f),..., h^+(f)}_{(n-r'')-{\rm times}}, -j_{r''},..., j_{1})]\in Q(I_<(f))_n.$$
for all positive integers $i_1\ge i_2\ge...\ge i_{r'}, j_1\ge...\ge j_{r''}$ such that $i_{r'}\ge h^-(f), -j_{r''}\le h^+(f)$. Consequently, $$\gamma({\rm cls}(r', r'', 0, \emptyset, \emptyset); n)\in Q(I_<(f))_{2n}.$$
We complete the proof by applying Lemma~\ref{Lqch}.\end{proof}

\section{The main result}\label{Smain}
We are now ready to state the general result. 
Consider a given simple highest weight module $L_\prec(f)$ where $f\in\mathbb F^\Theta$ is a $\prec$-locally constant and almost integral function. 
Let $\prec_f$ be a total order on $\Theta$ such that $\prec_f$ is $f$-equivalent to $\prec$, and $\prec_f$ is $f$-preferred with respect to a partition $\Theta_1\sqcup...\sqcup\Theta_t$ of $\Theta$, see Subsection~\ref{SShwuan}. 
Propositions~\ref{Prsz-}, \ref{Prsz+}, \ref{Prsz} imply that
$$I_{\prec_f}^{\Theta_1}(f)=I(r_1, 0, {X}, \emptyset),\hspace{10pt} I_{\prec_f}^{\Theta_t}(f)=I(r_t, 0, \emptyset, {Y}),$$
$$I_{\prec_f}^{\Theta_i}(f)=I(r_i, g_i, \emptyset, \emptyset), 1<i<t$$
for appropriate nonegative integers $r_1,..., r_t, g_1,..., g_t$ and Young diagrams $X, Y$.



\begin{theorem}\label{T71} We have $I_\prec(f)=I(r_1+...+r_t, g_2+...+g_{t-1}, {X}, {Y})$.\end{theorem}
\begin{proof} The proof follows the same lines as the proofs of Propositions~\ref{Prsz-},~\ref{Prsz+},~\ref{Prsz}. 
One first proves the inclusion
$$I(r_1+...+r_t, g_2+...+g_{t-1}, X, Y)\subset I_\prec(f)$$
by the same argument as above.

For the opposite inclusion, one considers functions $f_{i_1,..., i_r}$ which, restricted to $\Theta_i$, coincide with the corresponding functions constructed in the proofs of Propositions~\ref{Prsz-},~\ref{Prsz+},~\ref{Prsz}. 
This means that the integers $i_1,..., i_r$ arise as a union of independently chosen $t$-subsets of integers. 
With this modification, the argument goes though almost verbatim.
\end{proof}
\section{Examples}
\subsection{}Assume that $(\Theta, \prec)=(\mathbb Z_{>0}, <)$. Fix $\alpha\in\mathbb F, n\in\mathbb Z_{\ge1}$ and consider the $<$-locally constant function
\begin{equation}f:=(\underbrace{-1,..., -1}_{(n-1){\rm~times}}, \alpha, 0, 0, 0, 0,...)\label{Ef-1a0l}.\end{equation}
Then
$$h(f)=0,\hspace{10pt} f^+=(-1, -2, -3,..., -n, \alpha-n, -(n+1), -(n+2),...),$$
$$(f^+)^*=(..., (n+2), (n+1), n-\alpha, n,..., 3, 2, 1),$$
$${\rm\underline{seq}}((f^+)^*)= n-\alpha,\hspace{10pt}{\rm\overline{seq}}((f^+)^*)=(..., (n+2), (n+1), n, ..., 1).$$
Hence $r=1, {Y}=\emptyset,$ and 
$$I_<(f)=I(1, 0, \emptyset, \emptyset).$$

Assume next that $(\Theta, \prec)=(\mathbb Z_{<0}, <)$. 
Fix $\alpha\in\mathbb F, n\in\mathbb Z_{\ge1}$ and consider
\begin{equation}f:=(..., -1, -1, \alpha, \underbrace{0, ..., 0}_{(n-1)_{\rm~times}}),\label{Ef-1a0r}\end{equation}
Here
$$h(f)=-1,\hspace{10pt} f^+=(..., n, n-1, (n-1)+\alpha, (n-2),..., 2, 1, 0),$$
$${\rm\underline{seq}}((f^+)^*)= n-1+\alpha,\hspace{10pt}{\rm\overline{seq}}((f^+)^*)=(..., n, (n-1), (n-2), ..., 0).$$
Hence $r=1, {X}=\emptyset$, and again $$I_<(f)=I(1, 0, \emptyset, \emptyset).$$

Finally let $(\Theta, \prec)=(\mathbb Z, <)$. Fix $\alpha\in\mathbb F$ and consider
\begin{equation}f:=(..., -1, -1, \alpha, 0, 0,...)\label{Ef-1a0t}\end{equation}
where $f(n)=\alpha$. 
We have
$$h(f)=-1,\hspace{10pt} f^+=(..., 2-n, 1-n, 0-n, \alpha-n, -1-n, -2-n,...),$$
$${\rm\underline{seq}}((f^+)^*)= \alpha-n,\hspace{10pt}{\rm\overline{seq}}(f^+)=(..., 2-n, 1-n, 0-n, -1-n, -2-n, ...).$$
Hence $r=1, g=h(f)+r=0,$ and again $$I_<(f)=I(1, 0, \emptyset, \emptyset).$$

The above computations show that the simple highest weight modules with highest weights~(\ref{Ef-1a0l}),~(\ref{Ef-1a0r}),~(\ref{Ef-1a0t}) share the same annihilator in ${\rm U}(\frak{sl}(\infty))$, namely the primitive ideal $I(1, 0, \emptyset, \emptyset)$. 
Moreover, in~\cite{GP} it is proved that any simple nonintegrable highest weight $\frak{sl}(\infty)$-module with bounded weight multiplicities is isomorphic to one of the above highest weight modules. 
These highest weight modules are multiplicity free over $\frak h$. 
\subsection{}
Let $L$ be a direct limit of exterior powers $\Lambda^{k_n}(V_n)$ of $V_n$ where $V_n$ is a defining $\frak{sl}(n)$-module, and $k_n$ is a nondecreasing sequence satisfying $1< k_n<n$ and such that $k_{n+1}=k_n$ or $k_{n+1}=k_n+1$.
Assume that $$\lim k_n=\lim(n-k_n)=\infty.$$
Then one can show~\cite{GP} that $L$ is a highest weight $\frak{sl}(\infty)$-module for an appropriately chosen Borel subalgebra $\frak b(\prec)$, and that the highest weight $f\in\mathbb F^\Theta$ of $L$ can be chosen to take only values 1 and 0. 
Moreover, the Borel subalgebra $\frak b(\prec)$ can be chosen to be a Dynkin Borel subalgebra such that $(\Theta, \prec)=(\mathbb Z, <)$. 
Then, by Proposition~\ref{Prsz}, $$I_\prec(f)=I(0, 1, \emptyset, \emptyset).$$ 
In Section~\ref{Sback} we referred to $L$ as a semiinfinite fundamental representation of $\frak{sl}(\infty)$.  
\subsection{}
Consider the case when $\frak b$ is ideal with respect to a partition $\Theta_1\prec \Theta_2\prec \Theta_3$ of $S$ such that $\Theta_2$ is empty. 
Assume that $f$ has finitely many nonzero coordinates. 
The construction at the end of Subsection~\ref{SShwuan} provides an $f$-preferred partition $\Theta_1'\prec \Theta_2'\prec \Theta_3'$ of $\Theta$ for which $f|_{\Theta_2'}$ equals zero. 
This implies that
$$I_\prec^{\Theta_1'}(f)=I(r_1, 0, {X}, \emptyset),\hspace{10pt}I_\prec^{\Theta_2'}(f)=I(0, 0, \emptyset, \emptyset),\hspace{10pt}I_\prec^{\Theta_3'}(f)=I(r_3, 0, \emptyset, {Y})$$
for some $r_1, r_2\in\mathbb Z_{\ge0}$ and some Young diagrams ${X}, {Y}$. Therefore, $I_\prec(f)=I(r_1+r_3, {X}, {Y})$ by Theorem~\ref{T71}. 
It is easy to check that the primitive ideals obtained in this way run over all ideals of the form $I(r, 0, {X}, {Y})$ for arbitrary $r, {X}, {Y}$. 
The case when $r_1=r_3=0$ corresponds to the case of the simple tensor modules $V_{{X}, {Y}}$ mentioned in Section~\ref{Sback}.

\end{document}